\documentclass[12pt]{amsart}

\usepackage{kotex}
\usepackage{amsmath,amssymb,latexsym}
\usepackage{footnote}
\usepackage{fullpage}
\usepackage{etoolbox}
\usepackage{enumitem}
\usepackage[table]{xcolor}
\usepackage{color}
\usepackage[colorlinks,linkcolor=blue,citecolor=magenta, pagebackref]{hyperref}
\usepackage{soul}
\usepackage{microtype}
\usepackage{yhmath}
\usepackage{graphicx}
\usepackage{stackrel}
\usepackage{mathtools}
\usepackage{array}
\usepackage{float}
\usepackage{complexity}
\usepackage[sharp]{easylist}
\usepackage{tikz}
\usepackage[normalem]{ulem}
\usepackage[margin=1in]{geometry}
\usepackage{subcaption}
\usetikzlibrary{arrows.meta}

\usepackage[T1]{fontenc} 

\definecolor{FMcolor}{HTML}{DBE605}
\definecolor{SHcolor}{HTML}{50BCDF}
\definecolor{MHcolor}{HTML}{A383EA}
\definecolor{StoMcolor}{HTML}{FFB6C1}

\definecolor{arancio}{rgb}{1,0.6,0}

\usepackage[mathlines]{lineno}

\usepackage[colorinlistoftodos,bordercolor=orange,backgroundcolor=orange!20,linecolor=orange,textsize=scriptsize]{todonotes}
\newtheorem{theorem}{Theorem}[section]
\newtheorem{proposition}[theorem]{Proposition}

\newtheorem{lemma}[theorem]{Lemma}
\newtheorem*{lemma-nonumber}{Lemma}

\theoremstyle{definition}

\newtheorem{question}{Open Question}

\def\R{\mathbb{R}}
\def\Z{\mathbb{Z}}

\def\K{\mathsf{K}}
\def\L{\mathsf{L}}
\def\LL{\mathcal{L}}
\def\VV{\mathcal{V}}

\def\AA{\mathcal{A}}

\def\O{\mathcal{O}}
\def\M{\mathcal{M}}
\def\C{\mathcal{C}}

\def\sd{\operatorname{sd}}

\def\height{\operatorname{height}}
\def\st{\operatorname{st}}

\renewcommand{\leq}{\leqslant}
\renewcommand{\preceq}{\preccurlyeq}
\renewcommand{\succeq}{\succcurlyeq}

\newcommand{\under}[1]{\underline{#1}}

\newcommand{\rk}{\operatorname{rk}}

\usepackage{comment}

\title{Colorful circuits and colorful topes \\ in oriented matroids}

\author{Minho Cho}
\address{Minho Cho, School of Computational Sciences, Korea Institute for Advanced Study(KIAS), Seoul, South Korea.}
\email{minhocho.math@gmail.com}

\author{Seunghun Lee}% \orcidlink{0000-0003-0838-1680}}
\address{Seunghun Lee, Department of Mathematics, Keimyung University, Daegu, South Korea.}
\email{seunghun.math@gmail.com, seunghun.math@kmu.kr}

\thanks{The first author was supported by a KIAS Individual Grant (CG102101) at Korea Institute for Advanced Study. The second author was supported by the National Research Foundation of Korea (NRF) grant funded by the Korea government (MSIT) (RS-2026-25473635), by the Institute for Basic Science (IBS-R029-C1) and the Department of Mathematical Sciences of Korea Advanced Institute of Science and Technology (BK21).}

\author{Fr\'ed\'eric Meunier}
\address{Fr\'ed\'eric Meunier, CERMICS, ENPC, Institut Polytechnique de Paris, Marne-la-Vall\'{e}e, France.}
\email{frederic.meunier@enpc.fr}

\begin{document}

\begin{abstract}
	We provide a short proof of a conic version of the colorful Carath\'{e}odory theorem for oriented matroids.
	Holmsen's extension of the colorful Carath\'{e}odory theorem to oriented matroids (Advances in Mathematics, 2016) already encompasses several generalizations of the original result, but not its conic version.
	Our approach relies on a common generalization of Sperner's lemma and Meshulam's lemma---two closely related results from combinatorial topology that have found a number of applications in discrete geometry and combinatorics. This generalization may be of independent interest.
	
	Using a similar approach, we also establish the following colorful theorem for topes, whose special geometric case had not been considered before: {\em Given $n$ topes from a uniform oriented matroid with $n$ elements, if they agree on some element, then there is a way to select a distinct element from each tope, together with its sign, so as to form another tope of the oriented matroid.} Motivated by this theorem, we further explore other conditions leading to the same conclusion.
\end{abstract}



\maketitle

\section{Introduction}

One of the most beautiful theorems in discrete geometry is the colorful Carath\'{e}odory theorem by B\'ar\'any~\cite{barany1982generalization}. It states: {\em Given $d+1$ sets of points in $\R^d$, each of them containing the origin in its convex hull, there exists a way to select one point from each set so that the resulting set still contains the origin in its convex hull.} Many results of discrete geometry, especially those related to convexity and linear programming, have been abstracted into the setting of oriented matroids. This abstraction has often led to elegant and powerful generalizations, in which geometric arguments are recast in a purely combinatorial language, while retaining much of their original structure and intuition. In the present paper, we aim at exploring possible generalizations of and variations around the colorful Carath\'{e}odory theorem within the setting of oriented matroids. Oriented matroid theory is, however, a rich and technically demanding subject, and we assume throughout this introduction that the reader is familiar with its basic concepts. For completeness, these concepts are briefly recalled in Section~\ref{sec:om}.

A generalization of the colorful Carath\'{e}odory theorem for oriented matroids has already been proved by Holmsen~\cite{holmsen2016intersection} via a generalization of the topological colorful Helly theorem by Kalai and Meshulam~\cite{topological_helly}. The exact statement of Holmsen's theorem will be given in Section~\ref{subsec:discussion}, once the proper background will be provided. Actually, the result by Holmsen (as well as that by Kalai and Meshulam) is a generalization of a much stronger version of the colorful Carath\'{e}odory theorem, where the colors are replaced by some (non-oriented) matroid constraint, and where the condition is less demanding. Regarding this last point, instead of requiring that each color contains the origin in its convex hull, it is only required that the union of every two colors contains the origin in its convex hull. 

There exists a ``conic'' version of the colorful Carath\'{e}odory theorem  \cite{barany1982generalization}: {\em Given $d+1$ sets of vectors in $\R^{d+1}$, and a vector $p$ in $\R^{d+1}$ contained in the conic hull of each of these sets, there exists a way to select one vector from each set so that the resulting set still contains $p$ in its conic hull.} This conic version is also stronger than the original convex version  in the following sense: proving the convex version from the conic one is immediate, and there is no known direct derivation of the conic version from the convex one.  (Actually, we can be even more precise by considering the related algorithmic problems: there is a polynomial reduction from the convex problem to the conic one, and there is no known polynomial reduction in the other direction.)

A first contribution of this paper is a generalization of the conic version for oriented matroids, something which was absent from the literature. 

\begin{theorem}\label{thm:color-carac}
	Consider an oriented matroid of rank $r$. For every collection of $r$ positive circuits $C_1,\dots,C_r$ that share a common element $e$, there is a way to select at most one element from each $C_i\setminus \{e\}$ so as to form with $e$ a positive circuit.
\end{theorem}

We emphasize that a $C_i$ may share several elements with the resulting positive circuit. A very formal statement could finish by ``. . . {\em there exist a positive circuit $\widetilde C$ containing $e$ and an injective map $\pi\colon \widetilde C\setminus\{e\}\to[r]$ such that $f \in \underline{C}_{\pi(f)}$} for all $f \in \widetilde C\setminus\{e\}$.''

Theorem~\ref{thm:color-carac} generalizes the conic version of the colorful Carath\'{e}odory theorem stated above. Let $V$ be the matrix whose columns consist of all the vectors in the $d+1$ sets, together with $-p$. Let $E$ denote the set of columns of $V$, and let $e\in E$ be the element corresponding to $-p$. For every $x\in\ker V$, consider the sign vector in $\{+,-,0\}^E$ obtained by taking the sign of each coordinate of $x$. These sign vectors are the vectors of the oriented matroid $\M$ represented by $V$. The rank of $\M$ is at most $d+1$. Since $p$ belongs to the conic hull of each of the $d+1$ sets, each set gives rise to a positive circuit of $\M$ containing $e$: indeed, a positive dependence between $-p$ and the vectors of the set yields a sign vector in which all nonzero entries are positive. Thus, Theorem~\ref{thm:color-carac}, applied to these $d+1$ positive circuits, guarantees a positive circuit containing $e$ and containing at most one element from each of the $d+1$ sets. Translating back to the vector configuration, this means that one can select at most one vector from each set so that $p$ belongs to the conic hull of the selected vectors, which is precisely the conic colorful Carath\'{e}odory theorem.

The proof of Theorem~\ref{thm:color-carac}---given in Section~\ref{subsec:color-carac}---relies on a topological lemma, which turns out to be a simultaneous generalization of the Sperner lemma~\cite{sperner1928satz} and the Meshulam lemma~\cite{meshulam2001clique}, both lemmas which have found many applications in topological combinatorics, especially to get ``colorful'' combinatorial results (e.g., \cite{aharoni2007independent, Meshulam_domination}), and beyond~\cite{de2019discrete}. This generalization appears to be new and may be of independent interest. Florian Frick independently observed (personal communication) that it holds and can be proved using essentially the same argument as Meshulam's lemma. We state and prove this ``Sperner--Meshulam'' lemma in Section~\ref{subsec:sperner-meshulam}.

A generalization of the convex colorful Carath\'{e}odory theorem for oriented matroids can easily be obtained from Theorem~\ref{thm:color-carac} by a lifting argument. For sake of completeness, we state and prove this generalization. 
It is a special case of Holmsen's theorem, and this will be further discussed in Section~\ref{subsec:discussion}.

\begin{theorem}\label{thm:color_carac_convex}
	Consider an oriented matroid of rank $r-1$. For every collection of $r$ positive circuits $C_1,\dots,C_r$, there is a way to select at most one element from each $C_i$ so as to form a positive circuit.
\end{theorem}

A second contribution is Theorem~\ref{thm:n+1}, which shares a similar flavor with Theorem~\ref{thm:color-carac}. Yet, unlike the latter, even its special case for realizable oriented matroids is new. Before stating the theorem in its abstract form, we present its realizable version.

Consider a $d\times n$ real matrix $V$, with $d\leq n$, such that every $d\times d$ square submatrix is non-singular. As in the preceding discussion of how Theorem~\ref{thm:color-carac} generalizes the conic version of the colorful Carath\'{e}odory theorem, the sign vectors of the elements of $\ker V$ form the vectors of a realizable oriented matroid $\M$. While Theorem~\ref{thm:color-carac} concerns circuits, that is, vectors with minimal support, Theorem~\ref{thm:n+1} concerns objects at the other end of the spectrum: vectors with maximal support. In the terminology of oriented matroids, these vectors of maximal support are topes of the dual oriented matroid of $\M$.

Given elements $x^1,x^2,\ldots,x^n\in\ker V$ with full support and whose last component is positive, Theorem~\ref{thm:n+1} ensures that there exists an element $\tilde x \in\ker V$ with full support such that, for every $j\in[n]$, the coordinates $\tilde x_j$ and $x^i_j$ have the same sign for some $i\in[n]$, with each $x^i$ corresponding to exactly one coordinate.

We now state the theorem. Recall that an oriented matroid is \emph{uniform} if every circuit has size equal to its rank plus one. (In the realizable case, assuming each $d\times d$ square submatrix of $V$ being non-singular makes its associated matroid uniform, as well as its dual.)

\begin{theorem}\label{thm:n+1}
	Consider a uniform oriented matroid with $n$ elements. For every collection of $n$ topes $T_1,\dots,T_n$ such that all of them agree on a common element, there exist a tope $\widetilde T$ and a numbering $e_1,\ldots,e_n$ of the $n$ elements such that $\widetilde T$ and $T_i$ agree on $e_i$ for every $i$.
\end{theorem}

We emphasize that, in general, the tope $\widetilde T$ does not necessarily belong to the collection of the $n$ topes. The proof of Theorem~\ref{thm:n+1}---given in Section~\ref{subsec:n_and_r}---also relies on the Sperner--Meshulam-type lemma of Section~\ref{subsec:sperner-meshulam}.

While Theorems~\ref{thm:color-carac} and~\ref{thm:color_carac_convex} hold without any assumption on non-degeneracy, uniformity is required for Theorem~\ref{thm:n+1}, as shown by the following example. Consider the covectors of the matrix $\begin{psmallmatrix} 1 & 1 & 1 & 0 \\ 0 & 0 & 0 & 1 \end{psmallmatrix}$ (i.e., consider the image of its transpose and for each vector in the image, replace each coordinate by its sign). They form a rank-two oriented matroid on four elements, with
\[
(+,+,+,+)\, , \quad (-,-,-,-) \, , \quad (+,+,+,-) \, , \quad (-,-,-,+)
\]
as topes. Though, there is no $\widetilde T$ as in the theorem for $T_1=T_2=(+,+,+,+)$ and $T_3=T_4=(-,-,-,+)$. (This example can immediately be generalized to get a counter-example of any rank.)

One can ask whether the condition in Theorem~\ref{thm:n+1} that the topes agree on some element is necessary. 
Notice that, without adding extra conditions, just removing this condition is not possible. Indeed, the (uniform) oriented matroid with $n$ elements of rank one with only two cocircuits, $(+,+,\ldots,+)$ and $(-,-,\ldots,-)$, is a counter-example. However, finding other counter-examples does not look that easy and this leads us to formulate the following open question.

\begin{question} \label{conjecture_om}
	Given a collection of $n$ topes $T_1,\dots,T_n$ of a uniform oriented matroid on $n$ elements, under what conditions does there exist a tope $\widetilde T$ agreeing with each tope $T_i$ on a distinct element?
\end{question}

We did not succeed in answering this question in full generality. However, we were able to identify two sufficient conditions, in addition to that of Theorem~\ref{thm:n+1}, ensuring a positive answer to Question~\ref{conjecture_om}. One condition is that, among the $n$ topes, at most $r$ are distinct, where $r$ is the rank of the oriented matroid. It is stated as follows and the proof is given in Section~\ref{subsec:n_and_r} as well.

\begin{theorem}\label{thm_om_rank_r_general}
	Consider a uniform oriented matroid of rank $r$ and with $n$ elements. For every collection of $n$ topes $T_1,\dots,T_n$ among which at most $r$ are pairwise distinct, there exists a tope $\widetilde T$ agreeing with each $T_i$ on a distinct element.
\end{theorem}

An alternative and more formal way of stating this theorem could be: {\em Consider a uniform oriented matroid of rank $r$ on $[n]$. For every collection of topes $T_1,\ldots,T_r$ and positive integers $n_1, \ldots, n_r$ summing up to $n$, there exist a tope $\widetilde T$ and a partition $[n]=J_1\cup \cdots \cup J_r$ such that $|J_i| = n_i$ and $\widetilde T(j) = T_i(j)$ for all $i\in [r]$ and $j \in J_i$.} The original $n$ topes are obtained by considering $n_i$ copies of each $T_i$.

The rank of the oriented matroid being two is the second sufficient condition we have been able to identify. The proof---given in Section~\ref{sec:rank2}---relies on ``alternating words'' on the alphabet $\{+,-\}$.

\begin{theorem} \label{thm_om_rank_2}
	Consider a uniform oriented matroid of rank two on $n$ elements. For every collection of $n$ topes $T_1,\ldots,T_n$, there exists a tope $\widetilde T$ agreeing with each $T_i$ on a distinct element.
\end{theorem}

Recall  that the counter-example above, showing that extra conditions are indeed necessary to get a positive answer to Question~\ref{conjecture_om}, is of rank one. So, in the statement of Theorem~\ref{thm_om_rank_2}, it is not possible to replace ``of rank two'' by ``of rank at most two.''

\subsection*{Acknowledgments} The authors thank Eran Nevo for helping them write the argument relying on Alexander's duality in the proof of Lemma~\ref{lem:HH}. They are also grateful to Laura Anderson for fruitful discussions about triangulations of pseudosphere representations and fundamental facts of oriented matroids, and to Andreas Holmsen for explanations regarding the status of the colorful Carath\'{e}odory theorem within oriented matroids. Last but not least, they thank Michael Gene Dobbins for helpful discussion and comments regarding Open Question~\ref{conjecture_om}. The authors are also grateful to the reviewers whose comments help improve the quality of the manuscript.

\section{Background on oriented matroids} \label{sec:om}

This subsection is mostly based on the book by Bj\"{o}rner et al.~\cite{om_book}, which forms a standard reference, where many complementary facts about oriented matroids are available.

\subsection{Fundamental notions}\label{subsec:OMdef}

Oriented matroids are combinatorial structures that can be seen as collections of subsets of a ground set, where each element in a subset is ``signed.'' Therefore, to provide a proper definition, we start with a few notions and notation about such ``signed sets.''

A {\em signed subset} $X$ of a set $E$ is a pair $(X^+,X^-)$ of disjoint subsets of $E$. The {\em composition} $X \circ Y$ of two signed subsets $X,Y$ of $E$ is the signed subset of $E$ defined by $(X \circ Y)^+ \coloneqq X^+ \cup (Y^+\setminus X^-)$ and $(X \circ Y)^- \coloneqq X^- \cup (Y^-\setminus X^+)$. Given a signed subset $X$ of $E$, we denote the pair $(X^-, X^+)$ by $-X$ and its {\em support}, defined as $X^+ \cup X^-$, by $\underline{X}$.

Given a signed subset $X$ of $E$ and an element $e \in E$, we define
\[
X(e) \coloneqq \left \{ 
\begin{array}{rl} 
	+ & \text{if $e \in X^+$, } \\
	- & \text{if $e \in X^-$, } \\
	0 & \text{otherwise.} 
\end{array} \right .
\]
This provides a natural identification of each signed subset with an element of $\{+,-,0\}^E$. Such identification is used in the paper without further mention.

There is a natural partial order on the signed subsets: we write $X \preceq Y$ if $X^+ \subseteq Y^+$ and $X^- \subseteq Y^-$.

An {\em oriented matroid} $\M$ is a pair $(E,\VV)$ where $E$ is a finite set and $\VV$ is a collection of signed subsets of $E$ satisfying the following axioms:
\begin{enumerate}[label=(V\arabic*)]
	\item \label{empty} $(\varnothing,\varnothing) \in \VV$,
	\item \label{oppo} for every $X \in \VV$, we have $-X \in \VV$,
	\item \label{compo} for every $X,Y \in \VV$, we have $X \circ Y \in \VV$ 
	\item \label{elimin} for every $X,Y \in \VV$ and $e \in X^+ \cap Y^-$, there is a $Z \in \VV$ such that
	\begin{itemize}
		\item $Z^+ \subseteq (X^+ \cup Y^+)\setminus \{e\}$,
		\item $Z^- \subseteq (X^- \cup Y^-)\setminus \{e\}$,
		\item $(\underline{X} \setminus \underline{Y}) \cup (\underline{Y} \setminus \underline{X}) \cup (X^+ \cap Y^+) \cup (X^- \cap Y^-) \subseteq \underline{Z}$.
	\end{itemize}
\end{enumerate}
With this definition, $\VV$ is the set of {\em vectors} of the oriented matroid. Axiom~\ref{elimin} is the {\em vector elimination} axiom. There are many alternative ways to define oriented matroids, and we will comment on this a bit later, when other notions like circuits, covectors, topes, etc. are introduced. At this stage, we emphasize that the definition of an oriented matroid we give, based on vectors, is not the most standard one: indeed, simpler versions of the elimination axioms are often used and the validity of~\ref{elimin} is actually a result due to Edmonds, Fukuda, and Mandel~\cite{fukuda1982oriented,mandel1982topology}.
All in all, our definition of oriented matroids is a combination of Theorem~3.7.5 and Corollary~3.7.9 of the reference book by Bj\"{o}rner et al.~\cite{om_book}.

A signed subset $X$ is {\em positive} if $X^-=\varnothing$. The set of positive vectors is denoted by $\VV^+$. The vectors in $\VV \setminus \{(\varnothing,\varnothing)\}$ with supports that are minimal for inclusion  are the {\em circuits} of $\M$. The set of circuits is denoted by $\C$. Note that circuits can alternatively described as the minimal elements of $\VV \setminus \{(\varnothing,\varnothing)\}$ for the partial order $\preceq$. A circuit with a support of cardinality one is a {\em loop}. It is a useful fact that a signed subset of $E$ is a vector of $\M$ if and only if it is the composition of circuits \cite[Proposition 3.7.2]{om_book}.

The supports of the circuits of $\M$ form the family of circuits of an ordinary matroid, called the {\em underlying matroid}, on the same ground set~\cite[Remark 3.2.2 (1)]{om_book}. We define the {\em rank} of $\M$ as the rank of the underlying matroid, and it is denoted by $\rk(\M)$.

A useful operation we will implement at some places in this paper is that of adding a parallel copy of an element. Given an element $e$ that is not a loop, this consists in adding $e'$ to the ground set and in completing the set of signed circuits so that $e$ and $e'$ are parallel, i.e., $(\{e\},\{ e'\})$ is a signed circuit and every signed circuit containing $e$ gives rise to a signed circuit obtained by replacing $e$ with $e'$. The resulting collection of signed circuits defines an oriented matroid, which we refer to as the oriented matroid obtained by adding a parallel copy of $e$.

On the other hand, we can delete any subset $F \subseteq E$ of elements from $\M$, obtaining the \emph{restriction} of $\M$ to $E\setminus F$. It is an oriented matroid on $E \setminus F$, denoted by $\M \setminus F$, defined by the set of circuits $\{C \in \C \colon \underline{C} \subseteq E \setminus F\}$.

Given a subset $F \subseteq E$, the {\em contraction} of $\M$ by $F$ is the oriented matroid $\M / F$ on $E \setminus F$ whose circuits are 
the minimal elements (according to $\preceq$) of $\bigl\{(C^+ \setminus F, C^- \setminus F) \colon C \in \C\bigr\}$. (It is indeed an oriented matroid; see~\cite[Proposition~3.3.2]{om_book}.)

\subsection{Duality}\label{subsec:duality}
Every oriented matroid $\M$ can be associated with a ``dual'' oriented matroid, which is useful for many purposes. To define it, we require the following notion of orthogonality. Two signed subsets $X,Y$ of $E$ are {\em orthogonal} if $\underline{X} \cap \underline{Y} = \varnothing$ or there are $e,f \in \underline{X} \cap \underline{Y}$ such that $X(e)Y(e) = - X(f) Y(f)$. Define $\LL$ as the set of signed subsets of $E$  that are orthogonal to every vector of $\M$. It turns out that $(E,\LL)$ is also an oriented matroid, which we denote by $\M^\star$. More precisely, $\LL$ satisfies the axioms~\ref{empty},~\ref{oppo},~\ref{compo}, and~\ref{elimin} and is thus the vectors of $\M^\star$~\cite[Proposition~3.7.12]{om_book}. This oriented matroid $\M^\star$ is the {\em dual} oriented matroid of $\M$, the vectors of $\M^\star$ are the {\em covectors} of $\M$, and the circuits of $\M^\star$ are the {\em cocircuits} of $\M$. The set of positive covectors (i.e., the set of positive vectors of the dual oriented matroid) is denoted by $\LL^+$. The set of cocircuits is denoted by $\C^\star$. A cocircuit with a support of cardinality one is a {\em coloop}. The covectors with supports that are maximal for inclusion are the {\em topes} of $\M$. Note that topes can alternatively be described as the maximal elements of $\LL$ for the partial order $\preceq$.

It turns out that taking the dual and taking the underlying matroid ``commute'': the underlying matroid of $\M^\star$ is the dual of the underlying matroid of $\M$~\cite[Proposition~3.4.1]{om_book}. The rank of $\M^\star$ is thus equal to $n-\rk(\M)$.

\subsection{Uniform oriented matroids} Consider a uniform oriented matroid $\M$. As stated in the introduction, this means that every circuit is of size $\rk(\M)+1$. This actually implies that every subset of size $\rk(\M)+1$ is the support of a circuit (since it is dependent in the underlying matroid and cannot contain smaller independent set). This also implies that the underlying matroid is uniform and that the dual oriented matroid $\M^\star$ is uniform as well.

\subsection{Topological Representation Theorem}\label{subsec:topol}

One of the most important theorems about oriented matroids is the Topological Representation Theorem by Folkman and Lawrence~\cite{folkman-lawrence}, published in 1978. It provides a bridge between the combinatorial theory of oriented matroids and topology, making their combinatorial structure amenable to topological and geometric intuition.

A subset $S$ of the $d$-dimensional sphere $S^d$ is a {\em pseudosphere} if there is a homeomorphism $S^d \to S^d$ mapping the equator to $S$. Given a pseudosphere $S$ in $S^d$, the Jordan--Brouwer theorem ensures that $S^d \setminus S$ consists of two connected components such that $S$ is the boundary of each. These components are the {\em sides} of $S$.

Given a finite set $E$, a collection $(S_e)_{e \in E}$ of pseudospheres in $S^d$ forms an {\em arrangement of pseudospheres} if the following two conditions hold:
\begin{enumerate}[label=(\roman*)]
	\item for every $F \subseteq E$, the set $S_F \coloneqq \bigcap_{e \in F} S_e$ is homeomorphic to a sphere.
	\item for every $F \subseteq E$ and $e \in E$ such that $S_F \not\subseteq S_e$, the set $S_F \cap S_e$ is a pseudosphere in $S_F$ with sides $S_F \cap S_e^+$ and $S_F \cap S_e^-$, where $S_e^+$ and $S_e^-$ are the two sides of $S_e$.
\end{enumerate}

A crucial property of pseudosphere arrangements is given in the next proposition. This property, which is a result by Edmonds and Mandel~\cite{mandel1982topology}, is sometimes added in the definition of pseudosphere arrangements, even though this is unnecessary, as commented by Las Vergnas et al.~\cite[Section~5.1]{om_book}.

\begin{proposition}\label{prop:inters}
	The intersection of an arbitrary collection of closures of sides is homeomorphic either to a sphere or to a ball. 
\end{proposition}

An arrangement of pseudospheres is {\em signed} if there is fixed choice for each $S_e$ of {\em positive side}, denoted by $S_e^+$, and {\em negative side}, denoted by $S_e^-$. An arrangement of pseudospheres is {\em essential} if the intersection of all pseudospheres is empty. Given a signed arrangement $\AA = (S_e)_{e\in E}$ of pseudospheres, the following maps locate each point $x \in S^d$ with respect to the pseudospheres:
\[
\sigma_e(x) \coloneqq \left \{ 
\begin{array}{rl}
	+ & \text{if $x \in S_e^+$,} \\
	- & \text{if $x \in S_e^-$,} \\
	0 & \text{if $x \in S_e$.} \\
\end{array} \right.
\]
We denote then by $\sigma(x)$ the signed subset $(X^+,X^-)$ where $ X^+ = \{e \in E \colon \sigma_e(x) = +\}$ and $X^- = \{e \in E \colon \sigma_e(x) = -\}$, and we introduce the following collection of signed subsets of $E$
\[
\LL(\AA) \coloneqq \bigl\{\sigma(x) \colon x \in S^d\bigr\} \cup \{(\varnothing,\varnothing)\} \, .
\]

\begin{theorem}[Topological Representation Theorem] \label{thm_top_rep}
	Let $\LL$ be a collection of signed subsets of a finite set. Then $\LL$ is the set of covectors of a loop-free oriented matroid of rank $d+1$ if and only if $\LL = \LL(\AA)$ for some essential signed arrangement $\AA$ of pseudospheres in $S^d$.
\end{theorem}

Given a loop-free oriented matroid $\M$, a signed arrangement $\AA$ such that $\LL(\AA)$ is isomorphic to $\M$ is a {\em topological representation} of $\M$. The theorem ensures that every loop-free oriented matroid admits a topological representation.

Topological representations enjoy further properties, like the one stated in the next proposition. It relies on the map $\sigma$ introduced above.

\begin{proposition}[{\cite[Proposition~5.1.5]{om_book}}]\label{prop:regular}
	Let $\AA$ be an essential signed arrangement of pseudospheres. Then the set $\sigma^{-1}(X)$ for $X \in \LL(\AA) \setminus \{(\varnothing,\varnothing)\}$ are the open cells of a regular cell decomposition of $S^d$. Furthermore, the mapping
	\[
	X \longmapsto \{x \in S^d \colon \sigma(x) \preceq X \}
	\]
	gives an isomorphism of $\bigl(\LL(\AA) \setminus \{(\varnothing,\varnothing)\},\preceq\bigr)$ and the face poset of this regular cell decomposition.
\end{proposition}

\section{Proofs of Theorems~\ref{thm:color-carac},~\ref{thm:n+1}, and~\ref{thm_om_rank_r_general}}

The proofs of Theorems~\ref{thm:color-carac},~\ref{thm:n+1}, and~\ref{thm_om_rank_r_general} share similar ingredients. They all use the Sperner--Meshulam-type lemma mentioned in the introduction.
To apply this lemma, we also need connectivity results for specific simplicial complexes. We will establish these results using homology, with coefficients in $\Z_2$ throughout the paper without further mention. The topological representation theorem will be a key tool in obtaining these results.

\subsection{A ``Sperner--Meshulam'' lemma}\label{subsec:sperner-meshulam}

A simplicial complex is {\em homologically $k$-connected} if every reduced homology group vanishes from dimension $-1$ to dimension $k$. We remind that vanishing homology in dimension $-1$ means non-emptiness.

\begin{lemma}\label{lem:sperner-meshulam}
Consider a positive integer $k$ and a family $\{\K^I\}_{I \subseteq [k]}$ of simplicial complexes such that $\K^I$ is a subcomplex of $\K^{J}$ whenever $I \subseteq J$. Suppose that $\K^I$ is homologically $(|I|-2)$-connected for every  $I$, and assume we are given a labeling $\lambda$ of the vertices of $\K^{[k]}$ such that $\lambda(v) \in I$ whenever $v \in \K^I$. Then there exists a $(k-1)$-dimensional simplex of $\K^{[k]}$ whose vertices get distinct labels by $\lambda$.
\end{lemma}

Note that when $I$ and $J$ are disjoint so are $\K^I$ and $\K^J$ because of the labeling. The Sperner lemma is the special case when $\K^{[k]}$ is a triangulation of a standard simplex with $[k]$ as vertex set and $\K^I$ is the triangulation induced on the face of the standard simplex with $I$ as vertex set. The Meshulam lemma is the special case when $\K^I$ is the subcomplex of $\K^{[k]}$ induced by the vertices labeled with elements from $I$.
The proof will actually make clear that Lemma~\ref{lem:sperner-meshulam} actually holds under the weaker topological condition of requiring only that the $(|I|-2)$th homology groups of $\K^I$ vanishes. It is common knowledge that a similar remark holds for the original Meshulam lemma, but we decided to state our result in terms of connectivity, as usually done for this type of results.

\begin{proof}[Proof of Lemma~\ref{lem:sperner-meshulam}]
This proof relies heavily on the basic notions of chains and chain maps from algebraic topology. Chains are linear combinations of simplices in a given dimension, and chain maps are linear maps between chains. These objects lie at the foundation of algebraic topology, and especially to build homology, and all the necessary background can be found in the book by Munkres~\cite[Chapter~1, Sections~5 and~12]{munkres_algebraic_topology}. As noted at the beginning of this section, we use $\Z_2$-homology and this translates into chains with $\Z_2$-coefficients.

We start by showing the existence of a chain $c^I \in C_{|I|-1}(\K^{[k]},\Z_2)$ for each subset $I$ of $[k]$ such that $c^\varnothing = 1$ and, for every non-empty subset $I$ of $[k]$, the following two properties hold:
\begin{enumerate}[label=(\alph*)]
	\item \label{item:partial} $\partial c^I = \sum_{i \in I} c^{I \setminus \{i\}}$.
	\item \label{item:supp} $c^I$ is supported by $\K^I$.
\end{enumerate}
(When $|I|=1$, we see $\partial$ as the standard augmentation map used to define reduced homology.) To do so, we set $c^\varnothing \coloneqq 1$ and $c^{\{i\}} \coloneqq v_i$ where $v_i$ is an arbitrary vertex of $\K^{\{i\}}$. (Note that such a vertex exists since being $(-1)$-connected means being non-empty.) The relation $\partial c^{\{i\}} = c^\varnothing$ holds. Suppose now that for all $I$ of size at most $s \leq k-1$ the chains $c^I$ have been defined and that they satisfy the two properties~\ref{item:partial} and~\ref{item:supp}. Consider now $I' \subseteq [k]$ of size $s+1$. We have 
\[
\partial \sum_{i \in I'}c^{I' \setminus \{i\}} = \sum_{i \in I'} \partial c^{I' \setminus \{i\}} = \sum_{i \in I'}\sum_{i' \in I' \setminus \{i\}} c^{I' \setminus \{i,i'\}} = 0 \, ,
\]
where induction is used to get the second equality. The simplicial complex $\K^{I'}$ is homologically $(|I'|-2)$-connected.  In particular, its homology vanishes in dimension $|I'|-2$, which precisely means that for any chain $c \in C_{|I'|-2}(\K^{I'},\Z_2)$ such that $\partial c = 0$, there exists a chain $c' \in  C_{|I'|-1}(\K^{I'},\Z_2)$ such that $\partial c' = c$. There exists thus  a chain $c^{I'}$ supported by $\K^{I'}$, which implies~\ref{item:supp}, such that $\partial c^{I'} =\sum_{i \in I'}c^{I'\setminus \{i\}}$, which implies~\ref{item:partial}.

Then, we interpret the labeling $\lambda$ as a simplicial map from $\K^{[k]}$ to the standard simplex with $[k]$ as vertex set, which we denote by $\sigma^{[k]}$. Since $\lambda(v) \in I$ when $v \in V(\K^I)$ and since the $c^I$ are supported by $\K^I$, there exist $\alpha_I\in\Z_2$ such that $\lambda_{\sharp}(c^I) = \alpha_I \sigma^I$ for all $I \subseteq [k]$ (where $\sigma^I$ is the standard simplex with $I$ as vertex set). Note that $\alpha_{\{i\}} = 1$ for all $i \in [k]$. By the definition of the $c^I$ and the property of chain maps, we have $\partial \lambda_{\sharp}(c^I) = \sum_{i \in I}\lambda_{\sharp}(c^{I \setminus \{i\}})$. Since $\partial \sigma^I = \sum_{i\in I}\sigma^{I \setminus \{i\}}$, a direct induction shows that $\alpha_I = 1$ for every non-empty subset $I$ of $[k]$. There exists thus a $(k-1)$-dimensional simplex of $\K^{[k]}$ whose vertices get distinct labels by $\lambda$. This translates into the desired conclusion.
\end{proof}

\subsection{Further results on vectors and covectors in oriented matroids}\label{subsec:prel-conn}

In this subsection, we consider an oriented matroid $\M$ on $[n]$. We establish and present elementary results on $\VV$ and $\LL$---its sets of vectors and covectors---seen as partially ordered sets (with the usual order $\preceq$ on vectors and covectors of an oriented matroid) that will be used to establish Theorems~\ref{thm:color-carac},~\ref{thm:n+1}, and~\ref{thm_om_rank_r_general}. The proofs of this elementary results will rely on the topological representation Theorem~\ref{thm_top_rep}; see~\cite[Chapter~5]{om_book}.  Basic notions from combinatorial topology will be used, like homotopy equivalence, barycentric subdivision, contractible spaces, etc. Useful references are~\cite[Chapter~1]{using_borsuk} and the first two chapters of~\cite{munkres_algebraic_topology}. We also emphasize that because of the choice we made regarding the statements of the theorems---with circuits for Theorem~\ref{thm:color-carac} and with topes for Theorems~\ref{thm:n+1} and~\ref{thm_om_rank_r_general}---the topological representation theorem will be used sometimes for $\M$ and sometimes for its dual.

We denote by $\height(X)$ the {\em height} of a vector $X$, defined as the maximal size of a chain from the zero vector to $X$, minus one. (In fact, every maximal chain from zero to $X$ has the same size, because the poset $(\VV, \preceq)$ is {\em pure}~\cite[Theorem~4.1.14(i)]{om_book}.
For further details, we refer to~\cite[Section~4.1]{om_book}.) In particular, the height of the zero vector is zero and the height of any circuit is one.

\begin{proposition}\label{prop:height}
The inequality $|\under{X}| \leq \rk(\M) + \height(X)$ holds for every vector $X$.
\end{proposition}

\begin{proof}
The proof we propose relies on an application of the topological representation Theorem~\ref{thm_top_rep} to the dual oriented matroid. (A direct combinatorial proof is also possible.) This theorem requires the oriented matroid to be loop-free. Hence, we start by contracting all coloops of $\M$.
Call $\M'$ the oriented matroid obtained this way.
Notice that the rank decreases when coloops are contracted, but that  the support and the height of the vectors are not modified. Consider an essential topological representation of the dual of $\M'$. The dimension of the cell associated with a vector $X$ is equal to $\height(X) - 1$. The dimension of the sphere in the topological representation being $n - \rk(\M') - 1$, the cell $X$ is contained in at least $(n - \rk(\M') - 1) - (\height(X)-1) = n - \rk(\M') - \height(X)$ pseudospheres, i.e., the number of zeroes in $X$ is at least $n - \rk(\M') - \height(X)$. The size of $\under{X}$ is at most $\rk(\M') + \height(X)$. Since $\rk(\M') \leq \rk(\M)$, we get the desired inequality.
\end{proof}

We are going to state and prove a crucial lemma for the proof of Theorem~\ref{thm:color-carac}. For an element $e$ of $\M$ and a positive integer $h$, define $\VV^+_{e,h}$ as the set of positive vectors $X$ satisfying simultaneously the following two conditions:
\begin{itemize}
\item $C \preceq X\,\Longrightarrow\, C(e) = +$ for every circuit $C$.
\item $1 \leq \height(X) \leq h$.
\end{itemize}

\begin{lemma}\label{lem:HH}
If $\VV^+_{e,h}$ is not empty, then its order complex is homologically $(h-2)$-connected.
\end{lemma}

\begin{proof}
We assume that $\VV^+_{e,h}$ is not empty. In particular, $e$ is not a coloop. Since no vector contains a coloop, we can actually assume that $\M$ has no coloop. Take an essential topological representation of the dual of $\M$, where we denote as usual by $S_j$ the pseudosphere associated with an element $j$. In this proof, we work with the cell complex  $\Sigma$ induced by this topological representation. By default, we consider closed cells.

Consider the cell $\sigma$ corresponding to the composition of all positive circuits. Denote by $U^+$ the vertices of $\sigma$ that are not on $S_e$ (which is not empty because we assume that $\VV^+_{e,h}$ is not empty), and by $U^0$ the vertices that are on $S_e$. There are two possibilities. For each of them, we will consider the barycentric subdivision of a subcomplex of $\Sigma$; the cell complex $\Sigma$ is regular by Proposition~\ref{prop:regular}, and this implies that such a subdivision exists by~\cite[Proposition 4.7.8]{om_book}.

The first possibility is that $U^0 = \varnothing$. In this case, the order complex of $\VV^+_{e,h}$ is the barycentric subdivision of the $(h-1)$-skeleton of $\sigma$, and thus homologically $(h-2)$-connected.

\begin{figure}
	\begin{minipage}{0.45\linewidth}
		\begin{center}
			\includegraphics[scale=0.15]{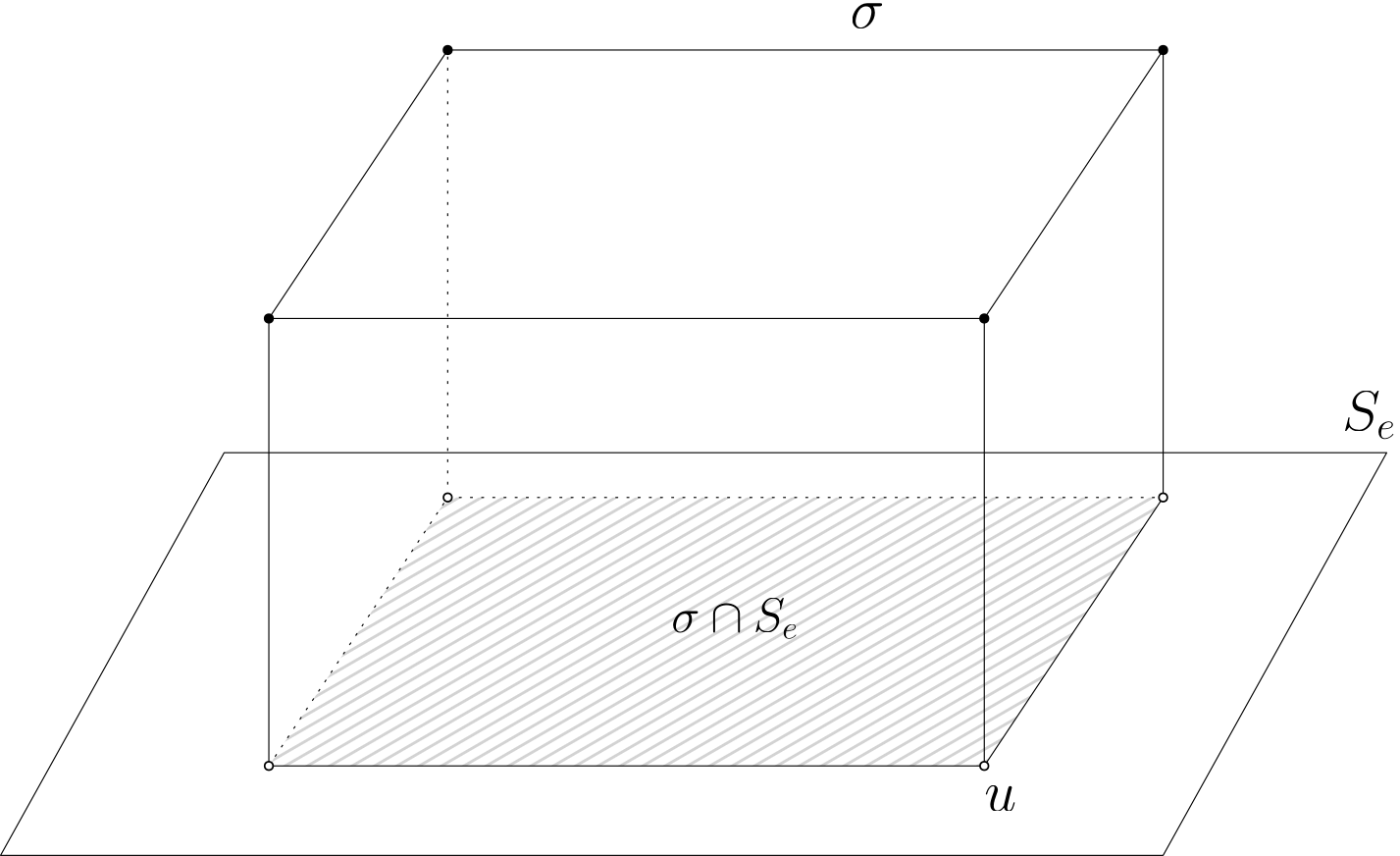}
			
		\end{center}
	\end{minipage}
	\begin{minipage}{0.45\linewidth}
		\begin{center}
			\includegraphics[scale=0.15]{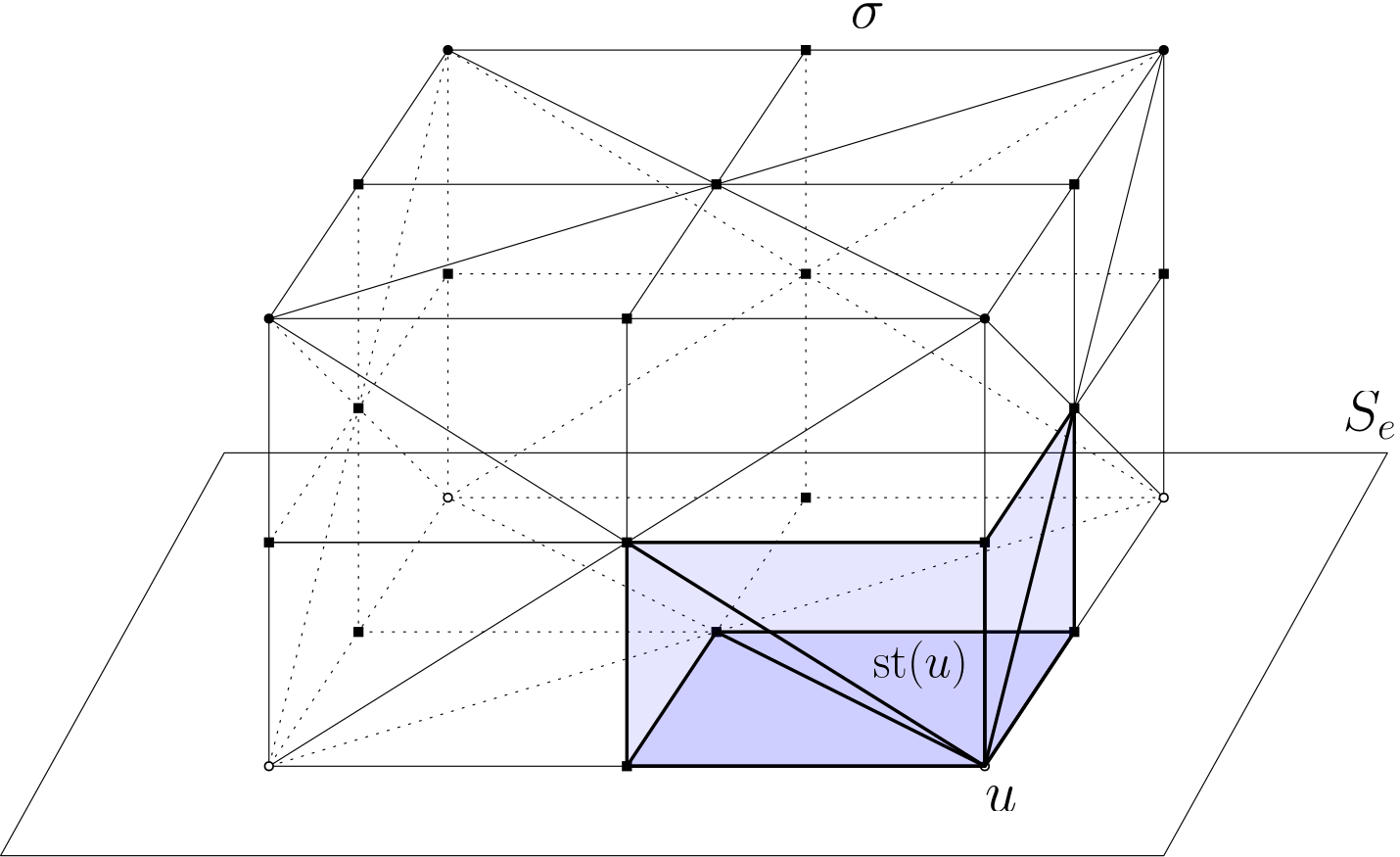}
			
		\end{center}
	\end{minipage}
	\caption{The cell $\sigma$ and a vertex $u \in U^0$ (left), and its star in $\sd(\partial \sigma)$ (right).
		\\ Vertices in $U^0$ are marked with empty circles, and the barycenters of cells in $\partial \sigma$ are marked with squares.
		\label{fig:placeholder}}
\end{figure}

The second possibility is that $U^0 \neq \varnothing$. Set $h' \coloneqq \min(h,\dim(\sigma))$.
The order complex of $\VV^+_{e,h}$ is the barycentric subdivision of the $(h'-1)$-skeleton of $(\partial\sigma)[U^+]$ (subcomplex of $\partial\sigma$ induced by $U^+$, that is, the subcomplex which consists of cells that only contain vertices from $U^+$). We finish the proof by establishing that $(\partial\sigma)[U^+]$ is acyclic (i.e., reduced homology groups vanish in every dimension): by the definition of homology for cell complexes, this implies immediately that the barycentric subdivision of its $(h'-1)$-skeleton is homologically $(h'-2)$-connected, and acyclic when $h > \dim(\sigma)$.

We take the barycentric subdivision of $\partial\sigma$, whose existence is guaranteed by  $\Sigma$ being a regular cell complex. 
We denote by $\sd(\partial\sigma)$ this barycentric subdivision. Consider the collection of stars $\st(u)$ for $u \in U^0$ where $\st(u)$ is the collection of all simplices of $\sd(\partial\sigma)$ whose union with $\{u\}$ is again a simplex of $\sd(\partial\sigma)$, that is, 
\[
\st(u)= \Bigl\{F \in \sd(\partial\sigma) \colon F\cup \bigl\{\{u\}\bigr\} \in \sd(\partial\sigma)\Bigr\} \, .
\]
(Note that this is the definition of stars as abstract complexes and not as geometric objects.) An example of such a star is displayed on Figure~\ref{fig:placeholder}.

We remark that for every non-empty subset $U' \subseteq U^0$, there is a unique minimal cell $\tau$ in $\partial \sigma$ containing $U'$. Indeed, the face poset of the subcomplex of $\Sigma$ restricted to $\sigma \cap S_e \subseteq \partial \sigma$ is a lattice~\cite[Corollary~4.3.7]{om_book} and the cell $\sigma \cap S_e$ contains $U'$. 

We explain now why $\bigcap_{u \in U'} \st(u)$ is a cone over $\tau$. Note that  $\bigcap_{u \in U'} \st(u)$ is non-empty because it contains $\tau$. Consider a simplex $F$ in $\bigcap_{u \in U'} \st(u)$. This simplex can be seen as a chain $\sigma_0 \subseteq \sigma_1 \subseteq \cdots \subseteq \sigma_k$ of cells in $\partial \sigma$. By definition of stars, adding $\{u\}$ at the bottom of this chain leads again to a simplex in $\bigcap_{u \in U'} \st(u)$. This means that $U'$ is contained in $\sigma_0$, and thus that $\tau$ is contained in $\sigma_0$, which in turn implies that adding $\tau$ to $F$ also leads to a simplex in $\bigcap_{u \in U'} \st(u)$.

The complex $\bigcap_{u \in U'} \st(u)$ being a cone, it is contractible. We can therefore apply the nerve theorem, in a version for simplicial subcomplexes~\cite[Theorem~6]{bjorner2003nerves}: homology of $\bigcup_{u \in U^0}\st(u)$ is that of the nerve of the stars $\st(u)$ for $u \in U^0$. Since $\bigcap_{u \in U^0}\st(u)$ contains the barycenter of the cell $\sigma\cap S_e$, it is non-empty, which means that the nerve of the stars is a full simplex, and this implies that $\bigcup_{u \in U^0}\st(u)$ is acyclic. We finish by the following chain of equalities and homotopy equivalence. We denote by $B$ the vertex set of $\sd(\partial\sigma)$ and by $B^0$ its subset corresponding to barycenters of cells hitting $U^0$:
\[
\|\partial\sigma\| \setminus \bigl\|\bigcup_{u \in U^0}\st(u)\bigr\| = \|\partial\sigma\| \setminus \|\sd(\partial\sigma)[B^0]\| \simeq \|\sd(\partial\sigma)[B \setminus B^0]\| = \|(\partial\sigma)[U^+]\| \, ,
\]
where the two equalities are immediate and where the homotopy equivalence results from one space being a strong deformation retract of the other~\cite[Lemma 4.7.27]{om_book}. By Alexander's duality~\cite[Theorem~71.1]{munkres_algebraic_topology},  we have $\widetilde H^k(\bigcup_{u \in U^0}\st(u)) \cong \widetilde H_{n-k-1}(\|\partial \sigma\| \setminus \|\bigcup_{u \in U^0}\st(u)\|)$ and since we are working with finite cell complexes and coefficients in a field, we have $\widetilde H^k(\bigcup_{u \in U^0}\st(u)) \cong \widetilde H_k(\bigcup_{u \in U^0}\st(u))$; see, e.g., the very end of Section~53 in the book by Munkres~\cite{munkres_algebraic_topology}. Hence, the acyclicity of $\bigcup_{u \in U^0}\st(u)$ implies  that of $\|\partial\sigma\| \setminus \bigl\|\bigcup_{u \in U^0}\st(u)\bigr\|$, and thus that of $(\partial\sigma)[U^+]$.
\end{proof}

The next two lemmas will be used in the proofs of Theorems~\ref{thm:n+1} and~\ref{thm_om_rank_r_general}. For $J^+, J^-$ two disjoint subsets of $[n]$, set 
\[
\LL_{J^+,J^-} \coloneqq \bigl\{X \in \LL \setminus \{0\} \colon X(j)\preceq + \text{ for every $j \in J^+$ and }  X(j)\preceq -  \text{ for every $j \in J^-$} \bigl\} \, .
\]

\begin{lemma}\label{lem:JJ}
If $J^+$ and $J^-$ are not both empty and $\LL_{J^+,J^-}$ contains a tope, then the order complex of $\LL_{J^+,J^-}$ is a ball. If $J^+$ and $J^-$ are both empty, then the order complex of $\LL_{J^+,J^-}$ is an $(\rk(\M)-1)$-dimensional sphere.
\end{lemma}

\begin{proof}
Without loss of generality, we can assume that $\M$ has no loop. Consider an essential topological representation of $\M$, where we denote as usual the pseudosphere associated with an element $j$ by $S_j$, and the two closed hemispheres with $S_j$ as boundary respectively by $\overline{S}_j^+$ and $\overline{S}_j^-$.
The topological representation theorem~ ensures the following.
\begin{itemize}
	\item When $J^+$ and $J^-$ are not both empty, the covectors in $\LL_{J^+,J^-}$ correspond to the cells situated in $\bigl(\bigcap_{j \in J^+} \overline{S}_j^+\bigr) \cap \bigl(\bigcap_{j \in J^-} \overline{S}_j^-\bigr)$,
	which, according to Proposition~\ref{prop:inters}, is homeomorphic either to a sphere of dimension at most $\rk(\M)-2$, or to a ball.
	Since it contains a tope, it has to be a ball.
	\item When $J^+$ and $J^-$ are both empty, the covectors in $\LL_{J^+,J^-}$ correspond to all cells in $S^{\rk(\M)-1}$.\qedhere
\end{itemize}
\end{proof}

The following lemma is a classical result from topological combinatorics. For applications and references to proofs of Quillen's lemma, see~\cite[Chapter 4]{om_book}. We denote the order complex of a poset $P$ by $\Delta(P)$.

\begin{lemma}[Quillen~\cite{quillen_lemma}]\label{lem:quillen}
Let $f \colon P \to Q$ be an order-preserving or order-reversing mapping of posets. Suppose that $\Vert\Delta(f^{-1}(Q_{\succeq x}))\Vert$ is contractible for all $x \in Q$. Then $f$ induces homotopy equivalence of $\| \Delta(P) \|$ and $\| \Delta(Q) \|$. 
\end{lemma}

\subsection{Proof of Theorem~\ref{thm:color-carac}}\label{subsec:color-carac}

\begin{proof}[Proof of Theorem~\ref{thm:color-carac}] If $e$ is a loop, then the statement is obvious. So assume that $e$ is not a loop. Up to making parallel copies of the elements, we can assume that $e$ is the only element of $\under{C}_i \cap \under{C}_j$ whenever $i \neq j$. Moreover, we can assume that the ground set is $\bigcup_{i=1}^r \under{C}_i$. (Note that now $\rk(\M)$ is not necessarily $r$, but certainly not bigger.)
For $I \subseteq [r]$, let $\K^I$ be the order complex of the set of positive vectors $X$ satisfying simultaneously the following three conditions:
\begin{itemize}
	\item $C \preceq X\,\Longrightarrow\, C(e) = +$ for every circuit $C$.
	\item $1 \leq \height(X) \leq |I|$.
	\item $\under{X} \subseteq \bigcup_{i\in I}\under{C}_i$.
\end{itemize}
It is homologically $(|I|-2)$-connected by Lemma~\ref{lem:HH}, applied to the oriented matroid restricted to the ground set $\bigcup_{i\in I}\under{C}_i$, with $h = |I|$. (Note that $\VV^+_{e, h}$ is non-empty since it contains any $C_i$.) Define a labeling $\lambda$ of the non-zero positive vectors of $\M$ by elements from $[r]$ as follows: for a  non-zero positive vector $X$, let $\lambda(X)$ be the index $i \in [r]$ such that $X$ selects the maximum number of elements of $C_i$; in case of a tie, choose the smallest possible $i$. This definition makes that $\lambda(X) \in I$ when $X \in \K^I$. (Here, we use that $e$ is not a loop.) The family $\K^I$ and the labeling $\lambda$ satisfy the conditions of Lemma~\ref{lem:sperner-meshulam}. Hence, there exist positive vectors $X_1 \prec \cdots \prec X_r$ such that $\lambda(X_i)$ get distinct values. 
By construction, the height of $X_r$ is at most $r$, which implies by Proposition~\ref{prop:height} that $|\under{X}_r| \leq 2r$ (Note that here we use $\rk(\M) \leq r$.) The element $e$ is in the support of all $X_i$ by definition of the $\K^I$. Thus, there is an $i^\star$ such that circuit $C_{i^\star}$ contributes in $X_r$ by at most one element (apart from $e$). There is a vector $X_j$ whose image by $\lambda$ is $i^\star$. Necessarily, no circuit $C_i$ contributes by more than one to this vector $X_j$, and since $X_1 \preceq X_j$, no circuit $C_i$ contributes by more than one to $X_1$. This latter vector $X_1$ being of height $1$, it is a positive circuit whose elements outside $e$ can be obtained by selecting at most one element from each $C_i\setminus\{e\}$.
\end{proof}

\subsection{Discussion around Theorem~\ref{thm:color_carac_convex}
}\label{subsec:discussion}

As already mentioned in the introduction, Theorem~\ref{thm:color_carac_convex} is an immediate consequence of Holmsen's theorem, and can also be proved easily from Theorem~\ref{thm:color-carac}. We remind Holmsen's theorem and explain how to deduce Theorem~\ref{thm:color_carac_convex}, and then we explain how to derive it from Theorem~\ref{thm:color-carac}.

\begin{theorem}[Holmsen's theorem~\cite{holmsen_colorful_caratheodory}]
Let $M$ be a matroid  and let $\O$ be an oriented matroid, both defined on the same ground set $E$ and satisfying $\rk(M) > \rk(\O)$. If every subset $S$ of $E$ such that $\rk_M(E \setminus S) < \rk(\O)$ contains a positive circuit of $\O$, then there exists a positive circuit of $\O$ contained in an independent set of $M$.
\end{theorem}

Consider $C_i$'s as in the statement of Theorem~\ref{thm:color_carac_convex}. Up to making parallel copies of the elements, we can assume that the supports of the $C_i$'s are pairwise disjoint and that the union of their supports is the ground set $E$ of the oriented matroid. Denote this latter by $\O$. Note that it is of rank at most $r-1$. Let $M$ be the partition matroid associated with the support of the $C_i$'s. Any subset $S$ of $E$ such that $\rk(M\setminus S) < \rk(\O)$ satisfies $\rk(M\setminus S) \leq r-2$, which means that it misses completely the support of at least two $C_i$'s. According to Holmsen's theorem, this implies that there exists a positive circuit of $\O$ formed by at most one element from each $C_i$'s, as desired.

We prove now Theorem~\ref{thm:color_carac_convex} from Theorem~\ref{thm:color-carac}, as announced in the introduction. The proof relies on a preliminary lemma. Its proof is omitted and follows directly from standard results; see in particular the first sections of Chapter~7 of the book by Bj\"{o}rner et al.~\cite{om_book}, and especially Propositions~7.1.4 and~7.2.4.

\begin{lemma}\label{lem:extension}
Let $\M$ be an oriented matroid on $E$ of rank $r-1$ and $e$ be an extra element not in $E$. Then there exists a rank $r$ oriented matroid $\widetilde{\M}$ on $E \cup \{e\}$ such that for every positive circuit $C$ of $\M$ the signed subset $(C\cup \{e\},\varnothing)$ is a (positive) circuit of $\widetilde{\M}$.
\end{lemma}

\begin{proof}[Proof of Theorem~\ref{thm:color_carac_convex} (based on Theorem~\ref{thm:color-carac})]
Apply Lemma~\ref{lem:extension} to the oriented matroid of the statement, and let $\widetilde{\M}$ be the obtained oriented matroid. Apply then Theorem~\ref{thm:color-carac} to the positive circuits $(C_i\cup\{e\},\varnothing)$. This shows that there is a way to select at most one element from each $C_i$ so as to form with $e$ a positive circuit of $\widetilde{\M}$. Contracting $e$ provides the desired positive circuit of the original oriented matroid.
\end{proof}

\subsection{Proof of Theorems~\ref{thm:n+1} and~\ref{thm_om_rank_r_general}} \label{subsec:n_and_r}

The proofs of Theorems~\ref{thm:n+1} and~\ref{thm_om_rank_r_general} rely on a specific simplicial complex which we will define now. This simplicial complex depends on a parameter $k$ that will be set to $n$ for the proof of Theorem~\ref{thm:n+1} and to $\rk(\M)$ for the proof of Theorem~\ref{thm_om_rank_r_general}.

Consider a uniform oriented matroid $\M$ on $[n]$ as in the statement of these theorems. Let $T_1,\ldots,T_k$ be topes. For a non-empty subset $I$ of $[k]$, we set 
\[
F^I \coloneqq \bigl\{(i_1,i_2,\ldots,i_n) \in I^n \colon (T_{i_1}(1), T_{i_2}(2),\ldots, T_{i_n}(n)) \in \LL\bigl\} \, .
\]
Let $\sigma^I$ be a geometric simplex whose vertices are identified with $I$. Then
define 
\[
(\sigma^I)^n \coloneqq \underbrace{\sigma^I \times \cdots \times \sigma^I}_{n \text{ terms}} \, .
\]
Note that $(\sigma^I)^n$ is a polytope (a.k.a. simplotope or prodsimplex) whose vertices are all possible $n$-tuples of elements in $I$. Some of these vertices are elements of $F^I$. We denote by $\L^I$ the polyhedral complex these vertices induce (keeping a face of $(\sigma^I)^n$ only if all its vertices are in $F^I$).

We illustrate this construction with $\M$ being the ``alternating oriented matroid'' of rank $3$ on $8$ elements.
For such an oriented matroid, topes correspond to sequences of $+$ and $-$ in which the sign changes at most twice. (We will come back to alternating matroids at the beginning of Section~\ref{sec:rank2}.) Figure~\ref{fig:complex} shows an example with $k=3$ topes. The highlighted cells form another tope, showing that $(2,3,3,1,2,3,1,2)$ belongs to $F^I$.
\begin{figure}
\[
\begin{array}{|c|cccccccc|}
	\hline
	j & 1 & 2 & 3 & 4 & 5 & 6 & 7 & 8 \\
	\hline
	T_1 & + & + & + & \cellcolor{cyan} - & - & - & \cellcolor{cyan} -  & + \\
	T_2 & \cellcolor{cyan} - & - & + & + & \cellcolor{cyan} + & + & - & \cellcolor{cyan} - \\
	T_3 & + & \cellcolor{cyan} - & \cellcolor{cyan} - & - & + & \cellcolor{cyan} + & + & + \\
	\hline
\end{array}
\]
\caption{Illustration of the construction of $F^I$.\label{fig:complex}}
\end{figure}

The proof of the theorems relies on the high connectivity of $\L^I$, as established by the following lemma. This is the place where the uniformity of $\M$ is needed.

\begin{lemma}\label{lem_connectivity_K^J}
Suppose that $\M$ is uniform, and let $I$ be a non-empty subset of $[k]$. Then $\L^I$ is $(r-2)$-connected, where $r$ is the rank of the oriented matroid. Moreover, if there is a $j$ in $[n]$ such that the $T_i(j)$ for $i\in I$ are all equal, then $\L^I$ is contractible.
\end{lemma}

\begin{proof}
Let $J^+ \coloneqq \bigcap_{i \in I} T_i^+$ and $J^- \coloneqq \bigcap_{i \in I} T_i^-$. Here $T_i^+$ and $T_i^-$ denote the sets of elements $j$ such that $T_i(j)=+$ and $T_i(j)=-$, respectively.
We are going to prove that $\L^I$ and the order complex $\Delta(\LL_{J^+,J^-})$ are homotopy equivalent. Since any $T_i$ with $i \in I$ belongs to $\LL_{J^+,J^-}$, applying Lemma~\ref{lem:JJ} will give the desired conclusion.

For $\tau = \tau_1 \times \tau_2 \times \cdots \times \tau_n$ a face of $\L^I$, let $X$ be defined by
\[
X(j) \coloneqq \left\{ 
\begin{array}{rll}
	+ & \text{if $T_i(j)= +$ for all $i \in \tau_j$,} \\[0.8ex]
	- & \text{if $T_i(j)= -$ for all $i \in \tau_j$,} \\[0.8ex]
	0 & \text{otherwise.}
\end{array}\right.
\]
We show that $X$ is actually a covector of $\M$ as follows. Every vertex of $\tau$ corresponds to a tope, i.e., every signed subset $Z$ with $\underline{Z} = [n]$ and satisfying $X^+ \subseteq Z^+$ and $X^- \subseteq Z^-$ is a tope. Take any $e \notin \underline X$. Using (co)vector elimination~\ref{elimin}, we get that every signed subset $Z$ with $\underline{Z} = [n] \setminus \{e\}$ and satisfying $X^+ \subseteq Z^+$ and $X^- \subseteq Z^-$ is a covector. Repeating this argument as many times as necessary, we get that $X$ is also a covector of $\M$.

Setting $f(\tau) \coloneqq X$ for every face $\tau$ of $\L^I$, this defines an order-reversing map  $f$ from the face  poset of $\L^I$ to $\LL_{J^+,J^-}$. Given $Y \in \LL_{J^+,J^-}$, define $\tau^Y = \tau_1^Y \times \tau_2^Y \times \cdots \times \tau_n^Y$ by setting $\tau_j^Y \coloneqq \sigma^I$ if $Y(j) = 0$ and by $\tau_j^Y \coloneqq \sigma^{\{i \in I \colon T_i(j)=Y(j)\}}$ otherwise.  Every signed subset $T$ such that $Y^+ \subseteq T^+$ and $Y^- \subseteq T^-$ is a tope. This is a consequence of $\M$ being uniform (and can be proved easily first by noting that every subset of size $n-\rk(\M)+1$ is the support of a cocircuit and  second by performing repeated compositions of $Y$ with well-chosen cocircuits). Thus $\tau^Y$ is a face of $\L^I$. The faces $\tau$ of $\L^I$ such that $f(\tau) \succeq Y$ are precisely the faces of $\tau^Y$. Since this latter is a polytope, it is contractible. The desired result follows then from the Quillen lemma (Lemma~\ref{lem:quillen}).
\end{proof}

The proofs of Theorems~\ref{thm:n+1} and~\ref{thm_om_rank_r_general} are then short and almost identical.

\begin{proof}[Proof of Theorem~\ref{thm:n+1}]
In this proof, we work with the complex $\L^{[n]}$, whose description is given at the beginning of the section. Note that this means that the parameter $k$ used in the description is here set to $n$. Let $\tau = \tau_1 \times \cdots \times \tau_n$ be a face of $\L^{[n]}$ of dimension at most $n-1$. Define $\lambda(\tau)$ as the smallest index $i$ such that there is at least one index $j$ with $\tau_j = \{i\}$. Such an $i$ exists because the dimension of $\tau$ is at most $n-1$, which means that there are at least one index $j$ such that $\tau_j$ is a singleton. Thanks to Lemma~\ref{lem_connectivity_K^J} and by the definition of $\lambda$, the condition of Lemma~\ref{lem:sperner-meshulam} is satisfied for $\K^I$ being the barycentric subdivision of the $(|I|-1)$-skeleton of $\L^I$. This means that there exist simplices $\tau^1 \subset \cdots \subset \tau^{n}$ of $\L^{[n]}$ such that $\lambda(\tau^i) = i$ for every $i \in [n]$, with $\tau^1$ of dimension $0$. The simplex $\tau^1$ corresponds thus to a vertex $(i_1,i_2,\ldots,i_n)$, where the $i_j$ are formed by the integers $1$ to $n$.
\end{proof}

\begin{proof}[Proof of Theorem~\ref{thm_om_rank_r_general}]
In this proof, we work with the complex $\L^{[r]}$, whose description is given at the beginning of the section. Note that this means that the parameter $k$ used in the description is here set to $r=\rk(\M)$. We actually prove the more formal statement given just after Theorem~\ref{thm_om_rank_r_general}: we consider $r$ topes $T_1,\ldots,T_r$, not necessarily distinct, and positive integers $n_1,\ldots,n_r$ summing up to $n$. Let $\tau = \tau_1 \times \cdots \times \tau_n$ be a face of $\L^{[r]}$ of dimension at most $r-1$. Define $\lambda(\tau)$ as the smallest index $i$ such that there are at least $n_i$ indices $j$ with $\tau_j = \{i\}$. Such an $i$ exists because the dimension of $\tau$ implies that there are at least $n-r+1 > \sum_{i=1}^r (n_i - 1)$ indices $j$ such that $\tau_j$ is a singleton. Thanks to Lemma~\ref{lem_connectivity_K^J} and by the definition of $\lambda$, the condition of Lemma~\ref{lem:sperner-meshulam} is satisfied for $\K^I$ being the barycentric subdivision of the $(|I|-1)$-skeleton of $\L^I$. This means that there exist simplices $\tau^1 \subset \cdots \subset \tau^r$ of $\L^{[r]}$ such that $\lambda(\tau^i) = i$ for every $i \in [r]$, with $\tau^1$ of dimension $0$. Setting $J_i$ as the indices $j$ such that $\tau_j^1 =\{i\}$ (where $\tau^1 = \tau_1^1 \times \cdots \times \tau_r^1$), we get the desired partition.
\end{proof}

\section{Alternating words and proof of Theorem~\ref{thm_om_rank_2}}\label{sec:rank2}

By a {\em binary word}, we mean a word on the alphabet $\{-,+\}$. Such a word is {\em alternating} if every two consecutive symbols are opposite. The {\em alternation number} of a binary word is the maximal length of an alternating subword (subsequence of non-necessarily consecutive symbols).

\begin{theorem}\label{thm:words}
Consider $n$ binary words $w^1, w^2, \ldots, w^n$ of length $n$. If each 
$w^i$ has alternation number at most two, then there exists a permutation $\pi$ of $[n]$ such that the binary word $w_1^{\pi(1)},w_2^{\pi(2)},\ldots,w_n^{\pi(n)}$ has alternation number at most two.
\end{theorem}

It is still very possible that versions of Theorem~\ref{thm:words} for other alternation numbers are true as well (but not for alternation number equal to one). This would contribute to Open question~\ref{conjecture_om}. Indeed, binary words of length $n$ with alternation number at most $r$, seen as $n$-tuples, are exactly the topes of the so-called {\em alternating oriented matroid} of rank $r$~\cite[Section~8.2 and Section~9.4]{om_book}. When the rank is two, the converse is also true: up to reorientation, every uniform oriented matroid of rank two is alternating. Theorem~\ref{thm_om_rank_2} is therefore just a reformulation of Theorem~\ref{thm:words}. 
Indeed, up to reorientation, there is essentially a unique topological representation of a uniform oriented matroid of rank two. Since the representation lies on a one-dimensional sphere, choose a point not belonging to any zero-dimensional pseudosphere, and orient each zero-dimensional pseudosphere so that the chosen point lies on its positive side. This yields the alternating orientation.

The remainder of the section is devoted to the proof of Theorem~\ref{thm:words}. The proof relies on two ingredients. The first one is a specific total order on the set of binary words of fixed length and with alternation number at most two. The second one is an elementary topological fact in dimension two.

\subsection{A total order on binary words with alternation number two}\label{subsec:order-two}

Consider two words $w,w'$ of length $n$ and with alternation number at most two. We set $w \preceq_b w'$ if we are in one of the three following (exclusive) situations:
\begin{itemize}
\item $w_n =+$ and $w_n' = -$.
\item $w_n = w'_n = +$ and $(w_i = - \Rightarrow w_i'= -)$ for all $i$.
\item $w_n = w'_n = -$ and $(w_i = + \Rightarrow w_i'= +)$ for all $i$.
\end{itemize}
It is immediate to check that this makes $\preceq_b$ a total order on the set of binary words of length $n$ and with alternation number at most two.

\subsection{Crossing on grid graphs}\label{subsec:grid}

In this section, we establish an elementary fact about the crossing of ``monotone'' curves on a grid graph. The statement and the proof are completely independent of the rest of the paper.

Consider an $(n+1) \times (n+1)$ grid digraph $D_n$ with ``monotone'' diagonals, where the vertices are the pairs $(i,j) \in [n+1]\times[n+1]$, and where there is an arc from a vertex $(i,j)$ to a distinct vertex $(i',j')$ when $i \leq i' \leq i+1$ and $j \leq j' \leq j+1$. An arc is {\em horizontal} if $j'=j$, it is {\em vertical} if $i'=i$, and it is {\em diagonal} otherwise. In a natural embedding, this terminology is consistent.

We introduce the ``toroidal'' version of $D_n$, which we denote $\overline{D}_n$. It is obtained from $D_n$ by identifying the vertices $(1,j)$ with the vertices $(n+1,j)$ and the vertices $(i,1)$ with the vertices $(i,n+1)$. Note in particular that the four pairs $(1,1)$, $(1,n+1)$, $(n+1,1)$, $(n+1,n+1)$ form a single vertex. With this identification, the graph $\overline{D}_n$ can actually be seen as embedded on a torus.

Given a directed walk and a directed cycle that are arc-disjoint, they {\em cross} whenever two consecutive arcs of the walk are separated at their common vertex by the cycle. On Figure~\ref{fig:crossing}, there is one crossing between the (blue) directed walk and the (red) directed cycle at column $i=9$ and row $j=4$. Crossings at vertices where the directed walk passes more than once are also possible.

We remind the reader that a {\em walk} in a directed graph is a sequence of arcs such that the head of an arc is the tail of the next arc. Moreover, a {\em closed walk} has the same definition, where the sequence is understood as a cyclic sequence.

\begin{proposition}\label{prop:grid}
Consider a closed directed walk $W$ of $\overline{D}_n$ that has at most $2n$ horizontal arcs and exactly $n$ vertical arcs. Then there exists a directed cycle $\Gamma$ of $\overline{D}_n$ using only diagonal arcs such that $W$ crosses $\Gamma$ at most once.
\end{proposition}

Proposition~\ref{prop:grid} is illustrated on Figure~\ref{fig:crossing}. Note that actually any directed cycle like $\Gamma$ in the statement, namely formed only by diagonal arcs, necessarily uses exactly $n$ arcs.

\begin{proof}[Proof of Proposition~\ref{prop:grid}]
If $W$ has no horizontal arc, the result is immediate. Suppose thus that $W$ has at least one horizontal arc.
Let's call a directed cycle using only diagonal arcs a {\em diagonal cycle}. There are exactly $n$ diagonal cycles. Since $W$ has at least one horizontal arc and at least one vertical arc, it has in particular a vertical arc followed by a horizontal arc. 
These two arcs cannot be involved together in any crossing.
The number of vertical arcs of $W$ being $n$, one diagonal cycle at least does not cross with $W$ at the head of a vertical arc. Call it $\Gamma$. We check that it satisfies the desired property, namely that it has at most one crossing with $W$.

Consider the universal cover of $\overline{D}_n$. The lifted version of $\Gamma$ is formed by copies of an infinite directed path bounding {\em diagonal strip regions}. The lifted version of $W$ is formed by copies of an infinite directed path. Pick an arbitrary copy, and keep from it a finite directed path of same length as $W$, which we denote by $\overline{W}$. We assume that $\overline{W}$ starts at a vertex at the interior of a diagonal region. 
Each crossing of $W$ with $\Gamma$ corresponds to $\overline{W}$ crossing the boundary of a diagonal strip region. By the choice of $\Gamma$, the directed path $\overline{W}$ can only cross such a boundary from the left to the right (i.e., by increasing the first coordinate). Since the number of horizontal arcs is at most $2n$ 
and the number of vertical arcs is exactly $n$, the directed path $\overline{W}$ ends either in the region where it starts, or in the region immediately on the right. Therefore, $\overline{W}$ crosses at most one boundary, which means that $W$ crosses $\Gamma$ at most once.
\end{proof}

\begin{figure}
\begin{tikzpicture}[scale=0.7, >=Stealth]
	\def\rows{8}
	\def\cols{12}
	
	% Vertices
	\foreach \i in {0,...,9} {
		\foreach \j in {0,...,9} {
			\node[circle, fill=black, inner sep=1pt] (v\i\j) at (\j,-\i) {};
		}
	}
	
	% Right arcs
	\foreach \i in {0,...,9} {
		\foreach \j in {0,...,8} {
			\draw[->] (v\i\j) -- (v\i\the\numexpr\j+1\relax);
		}
	}
	
	% Down arcs
	\foreach \i in {0,...,8} {
		\foreach \j in {0,...,9} {
			\draw[->] (v\i\j) -- (v\the\numexpr\i+1\relax\j);
		}
	}
	
	% Diagonal arcs
	\foreach \i in {0,...,8} {
		\foreach \j in {0,...,8} {
			\draw[->] (v\i\j) -- (v\the\numexpr\i+1\relax\the\numexpr\j+1\relax);
		}
	}
	
	% Blue walk
	\draw[->, ultra thick, blue] (v00) -- (v10);
	\draw[->, ultra thick, blue] (v10) -- (v11);
	\draw[->, ultra thick, blue] (v11) -- (v12);
	\draw[->, ultra thick, blue] (v12) -- (v13);
	\draw[->, ultra thick, blue] (v13) -- (v23);
	\draw[->, ultra thick, blue] (v23) -- (v33);
	\draw[->, ultra thick, blue] (v33) -- (v34);
	\draw[->, ultra thick, blue] (v34) -- (v35);
	\draw[->, ultra thick, blue] (v35) -- (v36);
	\draw[->, ultra thick, blue] (v36) -- (v37);
	\draw[->, ultra thick, blue] (v37) -- (v38);
	\draw[->, ultra thick, blue] (v38) -- (v39);
	% \draw[->, ultra thick, blue] (v39) -- (v49);
	
	\draw[->, ultra thick, blue] (v30) -- (v40);
	\draw[->, ultra thick, blue] (v40) -- (v41);
	\draw[->, ultra thick, blue] (v41) -- (v42);
	\draw[->, ultra thick, blue] (v42) -- (v52);
	\draw[->, ultra thick, blue] (v52) -- (v62);
	\draw[->, ultra thick, blue] (v62) -- (v63);
	\draw[->, ultra thick, blue] (v63) -- (v64);
	\draw[->, ultra thick, blue] (v64) -- (v65);
	\draw[->, ultra thick, blue] (v65) -- (v66);
	\draw[->, ultra thick, blue] (v66) -- (v76);
	\draw[->, ultra thick, blue] (v76) -- (v86);
	\draw[->, ultra thick, blue] (v86) -- (v87);
	\draw[->, ultra thick, blue] (v87) -- (v88);
	\draw[->, ultra thick, blue] (v88) -- (v89);
	% \draw[->, ultra thick, blue] (v89) -- (v99);
	\draw[->, ultra thick, blue] (v80) -- (v90);

	%Red cycle
	\draw[->, ultra thick, red] (v05) -- (v16);
	\draw[->, ultra thick, red] (v16) -- (v27);
	\draw[->, ultra thick, red] (v27) -- (v38);
	\draw[->, ultra thick, red] (v38) -- (v49);
	\draw[->, ultra thick, red] (v40) -- (v51);
	\draw[->, ultra thick, red] (v51) -- (v62);
	\draw[->, ultra thick, red] (v62) -- (v73);
	\draw[->, ultra thick, red] (v73) -- (v84);
	\draw[->, ultra thick, red] (v84) -- (v95);
	
\end{tikzpicture}
\caption{Illustration of Proposition~\ref{prop:grid}.\label{fig:crossing}}
\end{figure}

\subsection{Theorem~\ref{thm:words} as the existence of a ``diagonal''}

As in the statement of Theorem~\ref{thm:words}, consider $n$ binary words $w^1, w^2, \ldots, w^n$ of length $n$ and assume that each of them has alternation number at most two. Without loss of generality, we assume that they are indexed increasingly according to $\preceq_b$ (the total order defined in Section~\ref{subsec:order-two}). The next result ensures the existence of a binary word as promised by Theorem~\ref{thm:words}, with an extra ``diagonal'' structure. When larger than $n$, the index $i+k$ is identified with $i+k-n$.

\begin{figure}
\begin{tikzpicture}[>=Stealth]
	\def\cellsize{0.8}
	
	\foreach \i in {0,...,9} {
		\foreach \j in {0,...,9} {
			\node[circle, fill=gray!40, inner sep=1pt] (v\i\j) at (\j*\cellsize,\i*\cellsize) {};
		}
	}
	
	% Right arcs
	\foreach \i in {0,...,9} {
		\foreach \j in {0,...,8} {
			\draw[->, gray!40] (v\i\j) -- (v\i\the\numexpr\j+1\relax);
		}
	}
	
	% Down arcs
	\foreach \i in {1,...,9} {
		\foreach \j in {0,...,9} {
			\draw[->, gray!40] (v\i\j) -- (v\the\numexpr\i-1\relax\j);
		}
	}
	
	% Diagonal arcs
	\foreach \i in {1,...,9} {
		\foreach \j in {0,...,8} {
			\draw[->, gray!40] (v\i\j) -- (v\the\numexpr\i-1\relax\the\numexpr\j+1\relax);
		}
	}

	\foreach \j in {0,...,4} {
		\pgfmathsetmacro\x{(\j + 0.5)*\cellsize}
		\pgfmathsetmacro\y{8.5*\cellsize}
		\node at (\x,\y) {$-$};
	}
	\foreach \j in {5,...,8} {
		\pgfmathsetmacro\x{(\j + 0.5)*\cellsize}
		\pgfmathsetmacro\y{8.5*\cellsize}
		\node at (\x,\y) {$+$};
	}

	\foreach \j in {0,...,4} {
		\pgfmathsetmacro\x{(\j + 0.5)*\cellsize}
		\pgfmathsetmacro\y{7.5*\cellsize}
		\node at (\x,\y) {$-$};
	}
	\foreach \j in {5,...,8} {
		\pgfmathsetmacro\x{(\j + 0.5)*\cellsize}
		\pgfmathsetmacro\y{7.5*\cellsize}
		\node at (\x,\y) {$+$};
	}
	
	\foreach \j in {0,...,5} {
		\pgfmathsetmacro\x{(\j + 0.5)*\cellsize}
		\pgfmathsetmacro\y{6.5*\cellsize}
		\node at (\x,\y) {$-$};
	}
	\foreach \j in {6,...,8} {
		\pgfmathsetmacro\x{(\j + 0.5)*\cellsize}
		\pgfmathsetmacro\y{6.5*\cellsize}
		\node at (\x,\y) {$+$};
	}
	
	\foreach \j in {0,...,7} {
		\pgfmathsetmacro\x{(\j + 0.5)*\cellsize}
		\pgfmathsetmacro\y{5.5*\cellsize}
		\node at (\x,\y) {$-$};
	}
	\foreach \j in {7,...,8} {
		\pgfmathsetmacro\x{(\j + 0.5)*\cellsize}
		\pgfmathsetmacro\y{5.5*\cellsize}
		\node at (\x,\y) {$+$};
	}
	
	\foreach \j in {0,...,8} {
		\pgfmathsetmacro\x{(\j + 0.5)*\cellsize}
		\pgfmathsetmacro\y{4.5*\cellsize}
		\node at (\x,\y) {$-$};
	}
	
	\foreach \j in {0,...,1} {
		\pgfmathsetmacro\x{(\j + 0.5)*\cellsize}
		\pgfmathsetmacro\y{3.5*\cellsize}
		\node at (\x,\y) {$+$};
	}
	\foreach \j in {2,...,8} {
		\pgfmathsetmacro\x{(\j + 0.5)*\cellsize}
		\pgfmathsetmacro\y{3.5*\cellsize}
		\node at (\x,\y) {$-$};
	}
	
	\foreach \j in {0,...,1} {
		\pgfmathsetmacro\x{(\j + 0.5)*\cellsize}
		\pgfmathsetmacro\y{2.5*\cellsize}
		\node at (\x,\y) {$+$};
	}
	\foreach \j in {2,...,8} {
		\pgfmathsetmacro\x{(\j + 0.5)*\cellsize}
		\pgfmathsetmacro\y{2.5*\cellsize}
		\node at (\x,\y) {$-$};
	}
	
	\foreach \j in {0,...,3} {
		\pgfmathsetmacro\x{(\j + 0.5)*\cellsize}
		\pgfmathsetmacro\y{1.5*\cellsize}
		\node at (\x,\y) {$+$};
	}
	\foreach \j in {4,...,8} {
		\pgfmathsetmacro\x{(\j + 0.5)*\cellsize}
		\pgfmathsetmacro\y{1.5*\cellsize}
		\node at (\x,\y) {$-$};
	}
	
	\foreach \j in {0,...,3} {
		\pgfmathsetmacro\x{(\j + 0.5)*\cellsize}
		\pgfmathsetmacro\y{0.5*\cellsize}
		\node at (\x,\y) {$+$};
	}
	\foreach \j in {4,...,8} {
		\pgfmathsetmacro\x{(\j + 0.5)*\cellsize}
		\pgfmathsetmacro\y{0.5*\cellsize}
		\node at (\x,\y) {$-$};
	}

	\draw[->,ultra thick, blue] (4*\cellsize,9*\cellsize) to[bend left=30] (5*\cellsize,9*\cellsize);
	\draw[->,ultra thick, blue] (4*\cellsize,9*\cellsize) to[bend left=-30] (5*\cellsize,9*\cellsize);
	\draw[->,ultra thick, blue] (5*\cellsize,9*\cellsize) -- (5*\cellsize,8*\cellsize); 
	\draw[->,ultra thick, blue] (5*\cellsize,8*\cellsize) -- (5*\cellsize,7*\cellsize);
	\draw[->,ultra thick, blue] (5*\cellsize,7*\cellsize) -- (6*\cellsize,7*\cellsize);
	\draw[->,ultra thick, blue] (6*\cellsize,7*\cellsize) -- (6*\cellsize,6*\cellsize);
	\draw[->,ultra thick, blue] (6*\cellsize,6*\cellsize) -- (7*\cellsize,6*\cellsize);
	\draw[->,ultra thick, blue] (7*\cellsize,6*\cellsize) -- (7*\cellsize,5*\cellsize);
	\draw[->,ultra thick, blue] (7*\cellsize,5*\cellsize) -- (8*\cellsize,5*\cellsize);
	\draw[->,ultra thick, blue] (8*\cellsize,5*\cellsize) -- (9*\cellsize,5*\cellsize);
	\draw[->,ultra thick, blue] (0,5*\cellsize) -- (0,4*\cellsize);
	\draw[->,ultra thick, blue] (0,4*\cellsize) -- (\cellsize,4*\cellsize);
	\draw[->,ultra thick, blue] (\cellsize,4*\cellsize) -- (2*\cellsize,4*\cellsize);
	\draw[->,ultra thick, blue] (2*\cellsize,4*\cellsize) -- (2*\cellsize,3*\cellsize);
	\draw[->,ultra thick, blue] (2*\cellsize,3*\cellsize) -- (2*\cellsize,2*\cellsize);
	\draw[->,ultra thick, blue] (2*\cellsize,2*\cellsize) -- (3*\cellsize,2*\cellsize);
	\draw[->,ultra thick, blue] (3*\cellsize,2*\cellsize) -- (4*\cellsize,2*\cellsize);
	\draw[->,ultra thick, blue] (4*\cellsize,2*\cellsize) -- (4*\cellsize,\cellsize);
	\draw[->,ultra thick, blue] (4*\cellsize,\cellsize) -- (4*\cellsize,0);
	\draw[->,ultra thick, blue] (0,9*\cellsize) -- (\cellsize,9*\cellsize);
	\draw[->,ultra thick, blue] (\cellsize,9*\cellsize) -- (2*\cellsize,9*\cellsize);
	\draw[->,ultra thick, blue] (2*\cellsize,9*\cellsize) -- (3*\cellsize,9*\cellsize);
	\draw[->,ultra thick, blue] (3*\cellsize,9*\cellsize) -- (4*\cellsize,9*\cellsize);
	\draw[->,ultra thick, blue] (5*\cellsize,9*\cellsize) -- (6*\cellsize,9*\cellsize);
	\draw[->,ultra thick, blue] (6*\cellsize,9*\cellsize) -- (7*\cellsize,9*\cellsize);
	\draw[->,ultra thick, blue] (7*\cellsize,9*\cellsize) -- (8*\cellsize,9*\cellsize);
	\draw[->,ultra thick, blue] (8*\cellsize,9*\cellsize) -- (9*\cellsize,9*\cellsize);
\end{tikzpicture}
\caption{Illustration of the construction of the proof of Proposition~\ref{prop:diag}. Each row is formed by a word of length $9$. These words are ordered from the top to the bottom according to the $\preceq_b$ order.\label{fig:grid}}
\end{figure}

\begin{proposition}\label{prop:diag}
There exists a non-negative integer $k \leq n-1$ such that the binary word $w_1^{1+k},w_2^{2+k},\ldots,w_n^{n+k}$ has alternation number at most two.
\end{proposition}

\begin{proof}
We assume that there are at least two distinct $w^j$ since otherwise there is nothing to prove.

Consider the grid digraph $D_n$ of Section~\ref{subsec:grid}. Mark a horizontal arc $\bigl((i,j),(i+1,j) \bigl)$ whenever $w_i^{j-1} \neq w_i^j$ for $j \in \{2,\ldots,n+1\}$ (with $n+1 \coloneqq 1$). Mark a vertical arc $\bigl((i,j),(i,j+1) \bigl)$ whenever $w_{i-1}^j \neq w_i^j$ for $i \in \{2,\ldots,n\}$. The marked arcs form the ``boundary'' between the symbols $+$ and $-$ in the $n\times n$ matrix where the $(j,i)$ entry is $w_i^j$; see Figure~\ref{fig:grid}. 

We are going to mark extra vertical and horizontal arcs. For each $j$ such that $w^j$ uses only one symbol, mark the vertical arc $\bigl((1,j),(1,j+1) \bigl)$. Note that now there is exactly one vertical arc marked on each ``row,'' i.e., for each $j \in [n]$. The marked arcs form a subgraph 
$H$ of $\overline{D}_n$ where each vertex has equal in- and outdegree, except in two cases depicted hereafter. An {\em all-$-$ column} (resp.\ {\em all-$+$ column}) is an index $i$ such that $w_i^j=-$ (resp.\ $w_i^j=+$) for all $j$. There are only two cases where some vertices of $H$ may have different in- and outdegrees: 
\begin{enumerate}[label=(\alph*)]
	\item\label{minus} There is at least one all-$-$ column.
	\item\label{plus} There is no all-$-$ column, but there is at least one all-$+$ column.
\end{enumerate}
Assume that there is at least one vertex of $H$ with different in- and outdegrees. In cases~\ref{minus} and~\ref{plus}, the definition of the order makes that $H$ has a vertex of the form $(i,j)$ that has indegree $2$ and outdegree $0$, and a vertex of the form $(i',j)$ that has indegree $0$ and outdegree $2$. In case~\ref{minus}, $j$ is actually necessarily equal to $1$. In case~\ref{minus}, mark ``twice'' the horizontal arcs from $(i,1)$ to $(i',1)$, i.e., the arcs
\[
\bigl((i,1),(i+1,1)\bigr), \bigl((i+1,1),(i+2,1)\bigr),\ldots,\bigl((i'-1,1),(i',1)\bigr) \, ,
\]
where, as usual, the horizontal index is counted ``modulo $n$.''
In case~\ref{plus}, mark ``twice'' the horizontal arcs from $(i,j)$ to $(i',j)$, i.e., the arcs
\[
\bigl((i,j),(i+1,j)\bigr), \bigl((i+1,j),(i+2,j)\bigr),\ldots,\bigl((i'-1,j),(i',j)\bigr) \, .
\]
(In case~\ref{plus}, we necessarily have $i<i'$, with the convention that $i=n+1$ means $i=1$.) We add these arcs to $H$ (each arc marked ``twice'' being added twice), and $H$ has now equal in- and outdegree everywhere.

By construction, every column has exactly two marked horizontal arcs (counting twice the arcs marked ``twice'') and every row has exactly one vertical arc. Thus $H$ has $2n$ horizontal arcs and $n$ vertical arcs, and it is weakly connected (there is always a directed  path in $H$ from the head of a vertical arc to the tail of the vertical arc in the next row). The arcs of $H$ form then a closed directed walk $W$ as in the statement of Proposition~\ref{prop:grid}.
The directed cycle $\Gamma$ whose existence is ensured by this proposition translates into the desired ``diagonal'' binary word: this results from an easy case checking, keeping in mind that we only have to check that a sign change translates into a crossing between $W$ and $\Gamma$, and not necessarily the other way around since we are looking for an upper bound.
\end{proof}

\bibliographystyle{amsplain}
\bibliography{bibliography}

\end{document}